\documentclass[11pt]{article}
\usepackage[utf8]{inputenc}
\usepackage{amsmath,amssymb,amsthm,mathtools,mathrsfs}
\usepackage{geometry}
\usepackage[colorlinks]{hyperref}
\usepackage[backend=bibtex]{biblatex}
\usepackage{bbm}
\newcommand{\PP}{\mathbb{P}}
\newtheorem{theorem}{Theorem}[section]
\newtheorem{proposition}[theorem]{Proposition}

\newtheorem{remark}[theorem]{Remark}

\newcommand{\KL}{\operatorname{D_{KL}}}
\newcommand{\R}{\mathbb{R}}

\newcommand{\Prob}{\mathbb{P}}

\newcommand{\cQ}{\mathcal{Q}}
\newcommand{\cP}{\mathcal{P}}

\title{Robust Stochastic Outperformance\\ under Kullback--Leibler Ambiguity}
\author{Ozan Hur\\TUBITAK\\\texttt{ozan.hur@tubitak.gov.tr}\thanks{This work was conducted entirely outside of my employment using only personal resources, it does not represent the views of my employer, and no proprietary or confidential information was used. Any affiliations are for identification only and do not imply endorsement.}
}
\date{July~20,~2025}

\begin{document}
\maketitle

\begin{abstract}
We study the worst-case probability that \(Y\) outperforms a benchmark \(X\) when the law of \(Y\) lies in a Kullback-Leibler neighbourhood of the benchmark. The max--min problem over couplings admits a tractable dual (via optimal transport), whose optimiser is an exponential tilt of the benchmark law. The resulting solution reduces to a one-parameter family indexed by regulizer $\lambda$, which controls the KL information budget and induces an increasing transfer of probability mass from lower to higher outcomes. The formulation can be evaluated from the baseline distribution and may serve as a distribution-wide scenario generation environment as well as a basis for robust performance assessment.

\smallskip

\noindent\textbf{Keywords:} optimal transport; exponential tilting; stochastic dominance; stress testing.
\end{abstract}

\section{Problem statement}\label{sec:problem}
Our goal is to quantify \emph{robust stochastic outperformance} when the
reference law of one random variable is known exactly while the competing
quantity is only known up to a Kullback–Leibler (KL) divergence
constraint.

\paragraph{Stochastic outperformance.} 
Given two real-valued random variables \(X\) and \(Y\) defined on the same
probability space, let $\mathbb{P}$ be their joint law over $\mathbb{R}^2$. We say that \(Y\) \emph{outperforms} \(X\) if
\[
\mathbb{P}\,\!\bigl(Y>X\bigr)>\tfrac12.
\tag{1}
\]
The threshold \(1/2\) marks the point at which the event \(\{Y>X\}\) becomes
more likely than its complement and has been adopted in earlier work on
\textit{probability dominance} \cite{WratherYu1982}.  

\paragraph{Model setup.}
Let \((\Omega,\mathscr{F})\) be a measurable space and fix a benchmark
random variable \(X:\Omega\to\R\) whose law \(P\in\cP(\R)\) is fully
specified, with cumulative distribution function (c.d.f.) \(F\).
The competing random variable \(Y\) may follow \emph{any} distribution
\(Q\) within a KL ball of radius \(\varepsilon>0\) around \(P\),
\begin{equation}\label{eq:KLball}
  \cQ_{\varepsilon}
  :=\bigl\{Q\ll P : \KL(Q\Vert P)\le \varepsilon\bigr\},
  \qquad
  \KL(Q\Vert P)=\int \log\frac{dQ}{dP}\,dQ.
\end{equation}
The KL ball captures statistical ambiguity such that larger \(\varepsilon\) allow
greater deviations from the benchmark law, while \(\varepsilon=0\) forces
\(Q=P\).

\paragraph{Robust outperformance criterion.} Because we seek a \emph{robust} guarantee of outperformance, we compute the
\emph{worst‑case} probability of $Y$ exceeding $X$ over \emph{all} couplings
that share the given marginals.  For $P,Q\in\cP(\R)$ let
\[
  \Pi(P,Q):=\Bigl\{\gamma\in\cP(\R^{2}):\gamma\circ X^{-1}=P,\;
                                   \gamma\circ Y^{-1}=Q\Bigr\},\qquad
  \Prob_{\gamma}(Y>X)=\int_{\R^{2}} I_{\{y>x\}}\,d\gamma(x,y).
\]
Our performance metric is the max–min value
\begin{equation}\label{eq:robust-problem}
  V_{\varepsilon}
  :=\max_{Q\in\cQ_{\varepsilon}}
    \;\min_{\gamma\in\Pi(P,Q)}
    \Prob_{\gamma}(Y>X)>1/2,
\end{equation}
where $\cQ_{\varepsilon}$ is the KL ball \eqref{eq:KLball}.  The inner
minimisation is an optimal‑transport problem with cost
$c(x,y)=I_{\{y>x\}}$.  A distribution $Q^{\star}$ attaining the outer
maximum is the \emph{least‑favourable} law for $Y$ within the KL ball,
while a coupling $\gamma^{\star}$ attaining the inner minimum realises the
most adversarial dependence structure between $X$ and $Y$.  In this way
$V_{\varepsilon}$ provides a lower bound on the probability of
outperformance that is immune to both distributional ambiguity and
dependence uncertanity.

\paragraph{Interpretation of robustness.}
Expression \eqref{eq:robust-problem} answers the question:
\emph{“What is the largest worst-case probability with which \(Y\) can be guaranteed
to outperform \(X\) when nature may choose any distribution within
\(\cQ_{\varepsilon}\)?”}  The outer maximiser
selects the marginal of \(Y\) that  that most favours outperformance, while the inner minimiser selects the scenario that make
\(Y\le X\) as likely as possible. 

\paragraph{Stochastic dominance perspective.}
The least-favourable choice of the distribution for $Y$ that achieves our worst-case guarantee can be read as an \emph{order improvement} relative to the benchmark law of $X$. In particular, the worst-case law puts \emph{more probability on higher outcomes} than the benchmark, i.e., its CDF lies everywhere \emph{below} the benchmark which is exactly first-order stochastic dominance (FSD). This reflects the explicit exponential-tilt solution, which moves probability mass upward an “increasing transfers’’ view of FSD \cite{Mueller2013Transfers}.
 Thus, the robust outperformance is not only a worst-case probability bound but also identifies a marginal that is “better in the FSD sense” than the baseline law, providing an order-based, distribution-wide perturbation.

\section{Explicit solution by exponential tilting}\label{sec:solution}

In line with \cite[Sec.~2]{Kamihigashi2020Partial}, we provide the next proposition.
\begin{proposition}[Dual formula]\label{prop:dual}
Let \(P,Q\in\mathcal P(\mathbb R)\) have cumulative distribution functions \(F\) and \(G\).
Then the optimal transport cost
\[
   T(P,Q)
   \;:=\;
   \min_{\gamma\in\Pi(P,Q)}
   \int_{\mathbb R^2}\!
        I_{\{y>x\}}\,
   \mathrm d\gamma(x,y)
\]
admits the one–dimensional representation
\[
   T(P,Q)=\max_{a\in\mathbb R}\bigl[F(a)-G(a)\bigr].
\]
\end{proposition}

\begin{proof}
The cost function \(c(x,y)= I_{\{y>x\}}\) is bounded and lower–semicontinuous on the Polish space \(\mathbb R\times\mathbb R\). Hence, by \cite[Thm.~4.10]{Villani2009}, an optimal coupling \(\gamma^\star\in\Pi(P,Q)\) exists that attains the infimum in the primal problem. Choose the dual potential \(\psi(y)=I_{\{y\le a\}}\). Because
\(c(x,y)\le\zeta(x)+\phi(y)\) with the integrable bounds \(\zeta\equiv1\) and \(\phi\equiv0\), the conditions of \cite[Thm.~5.10\,(iii)]{Villani2009} are satisfied. Consequently, the Kantorovich dual supremum is attained, and we may replace \(\sup\) by \(\max\). To compute the \(c\)-transform of \(\psi\), observe that
\[
  \psi^{c}(x)
  :=\inf_{y\in\mathbb R}\bigl\{\psi(y)+c(x,y)\bigr\}
  =\inf_{y\in\mathbb R}
    \bigl(I_{\{y\le a\}}+I_{\{y>x\}}\bigr)
  =I_{\{x\le a\}},
\]
where the last equality follows by distinguishing the cases \(x\le a\) and \(x>a\).
Substituting the pair \((\psi,\psi^{c})\) into the dual formulation yields the claimed one–dimensional representation.
\end{proof}

Based on the previous propositon, the optimisation problem \eqref{eq:robust-problem} can be written as
\[
   V_{\varepsilon}
   =\max_{Q\in\mathcal Q_{\varepsilon}}
     \max_{a\in\mathbb R}\,[F(a)-G(a)].
\]
Moreover, the order of maximisation may be interchanged because both variables are
optimised over a joint product domain,
\begin{equation}\label{eq:dual}
   V_{\varepsilon}
   =\max_{a\in\mathbb R}
     \max_{Q\in\mathcal Q_{\varepsilon}}
     [F(a)-G(a)].
\end{equation}

\begin{proposition}[Exponential-tilt solution of the inner problem]%
\label{prop:tilt}
Fix \(a\in\mathbb R\) and consider the programme
\begin{equation}\label{eq:inner}
   \min_{Q\ll P}\bigl\{Q((-\infty,a]):
           D_{\mathrm{KL}}(Q\Vert P)\le\varepsilon\bigr\}.
\end{equation}
\begin{enumerate}
\item[(i)]
      There exists a unique solution and it is an \textbf{exponentially tilted}
      measure
      \begin{equation}\label{eq:Qtilt}
         \frac{\mathrm dQ_{\lambda,a}}{\mathrm dP}(x)=
         \frac{\exp\!\bigl[-I_{\{x\le a\}}/\lambda\bigr]}
              {Z(\lambda,a)},
         \qquad
         Z(\lambda,a)=1-\bigl(1-e^{-1/\lambda}\bigr)F(a),
      \end{equation}
      where \(\lambda>0\) is the Lagrange multiplier.
\item[(ii)] 
      Under \(Q_{\lambda,a}\),
      \begin{equation}\label{eq:G-lambda}
         G_{\lambda,a}(a)=
         \frac{e^{-1/\lambda}F(a)}{Z(\lambda,a)},
         \qquad
         D_{\mathrm{KL}}\!\bigl(Q_{\lambda,a}\Vert P\bigr)=
         -\log Z(\lambda,a)-\frac{1}{\lambda}\,G_{\lambda,a}(a).
      \end{equation}
\item[(iii)] 
      For each fixed \(a\) the map
      \(\lambda\mapsto D_{\mathrm{KL}}(Q_{\lambda,a}\Vert P)\)
      is continuous and strictly decreasing on \((0,\infty)\).
      Hence, for every \(\varepsilon>0\) there is a unique
      \(\lambda(a)\) such that
      \(D_{\mathrm{KL}}(Q_{\lambda(a),a}\Vert P)=\varepsilon\).
\end{enumerate}
\end{proposition}

\begin{proof}
Introduce the Lagrangian
\[
   L(Q;\lambda):=
   Q((-\infty,a])+\lambda\bigl(D_{\mathrm{KL}}(Q\Vert P)-\varepsilon\bigr),
   \qquad \lambda\ge0.
\]
Because the indicator loss \(I_{\{x\le a\}}\) is bounded, it satisfies
the light-tail Assumption 1 of \cite{HuHong2012}.  
Consequently, \cite[Thm.~2]{HuHong2012} guarantees existence of a minimiser,
and \cite[Prop.~1]{HuHong2012} shows that the minimiser must be of the
exponential-tilt form \eqref{eq:Qtilt}, proving \textit{(i)}.
Substituting \eqref{eq:Qtilt} into the definitions of \(G\) and
\(D_{\mathrm{KL}}\) yields \eqref{eq:G-lambda}, giving \textit{(ii)}.
Finally, differentiating the right-hand side of \eqref{eq:G-lambda} with
respect to \(\lambda\) shows the derivative is negative, establishing the
monotonicity in \textit{(iii)}.
\end{proof}

\begin{remark}[Reduction to a single-variable maximisation]
For each \(a\) choose the unique \(\lambda(a)\) from
Proposition~\ref{prop:tilt}\,\textit{(iii)} and set
\(Q^{a}:=Q_{\lambda(a),a}\).  Then \eqref{eq:dual} reduces to maximising
\[
   \phi_{\lambda(a)}(a)
   :=F(a)-G_{\lambda(a),a}(a)
\]
over \(a\in\mathbb R\).
\end{remark}

\begin{remark}[Computational motivation for the fixed-\texorpdfstring{$\lambda$}{lambda} dual]\label{rem:fixed-lambda}
In the radius‐constrained formulation, each candidate threshold \(a\) requires
solving an inner problem to find the unique multiplier \(\lambda(a)\) satisfying
\(\KL\!\bigl(Q_{\lambda,a}\Vert P\bigr)=\varepsilon\).
Although each inversion is inexpensive, repeating it over a fine grid of
\(a\)-values renders the outer maximisation computationally burdensome.\\

A widely used remedy, see \cite{BreuerCsiszar2016,GlassermanXu2014} is to fix a
single multiplier \(\lambda>0\) \emph{once} and dispense with the inner loop.
The resulting dual reduces to the unconstrained maximisation of
\[
   \phi_\lambda(a)
   \;:=\;
   F(a)-G_{\lambda,a}(a)
   \;=\;
   F(a)\,
   \frac{(1-F(a))(1-e^{-1/\lambda})}
        {1-F(a)+e^{-1/\lambda}F(a)},
   \qquad a\in\mathbb R.
\]
\end{remark}

\begin{proposition}[Existence (and uniqueness) of a maximiser of $\phi_\lambda$]%
\label{prop:a-star}
Let \(F\) be the cumulative distribution function of \(P\) and fix \(\lambda>0\).

\begin{enumerate}
\item[(i)] There exists at least one maximiser \(a^\star\in\mathbb R\)
      of \(\phi_\lambda\).
\item[(ii)] If \(F\) is strictly increasing, then the maximiser is unique.
\end{enumerate}
\end{proposition}

\begin{proof}
Set \(C:=1-e^{-1/\lambda}\in(0,1)\) and put \(x:=F(a)\in\mathrm{Range}(F)\subseteq[0,1]\).
Then \(\phi_\lambda(a)=h(x)\) with
\[
   h(x):=\frac{C\,x(1-x)}{1-Cx}, 
   \qquad 0\le x\le 1.
\]
A direct derivative calculation shows that
\(h'(x)=0\) iff \(x=x^\star:=(1+e^{-1/(2\lambda)})^{-1}\in(1/2,1)\),
and \(h''(x^\star)<0\); hence \(h\) attains its global maximum at \(x^\star\).

\begin{enumerate}

\item \textit{Existence.}
The range of any CDF \(F\) is compact—either the full interval \([0,1]\) in the
continuous case, or a finite/denumerable compact subset in the discrete case.
Because \(h\) is continuous on \([0,1]\), the supremum of
\(h(x)\) over \(\mathrm{Range}(F)\) is attained.
If \(x^\star\in\mathrm{Range}(F)\), take \(a^\star:=F^{-1}(x^\star)\);
otherwise choose any \(a^\star\) with \(F(a^\star)\) achieving
\(\sup_{x\in\mathrm{Range}(F)}h(x)\).

\item \textit{Uniqueness.}
When \(F\) is strictly increasing, its range is the entire \([0,1]\), so
\(x^\star\) is attained uniquely at \(a^\star=F^{-1}(x^\star)\).
\end{enumerate}
\end{proof}

\paragraph{Selection of multiplier \texorpdfstring{$\lambda$}{lambda}.}
The stochastic-outperformance objective~\eqref{eq:robust-problem} is of practical interest only when the maximum value of the objective $\phi(a)$ can exceed one half, indicating a meaningful separation. This requires a careful selection of the fixed multiplier $\lambda$. Hence, we focus on values of $\lambda$ that ensure the problem is non-trivial, namely
\[
  \lambda
  \;>\;
  \lambda_{\mathrm{crit}}
  :=\inf\Bigl\{\lambda>0:
       \max_{a\in\R}\phi(a)>\tfrac12\Bigr\},
\]
so that for the corresponding maximiser $a^{\ast}$ we have
$\phi(a^{\ast})>\tfrac12$.
If $\lambda\le\lambda_{\mathrm{crit}}$, the regularisation imposed by the multiplier is too strong, and no threshold $a$ can achieve the required level of outperformance. In practice, an admissible value $\lambda > \lambda_{\mathrm{crit}}$ can be found efficiently by testing several candidate values.

\begin{proposition}[First-order dominance of the $\lambda$–tilted optimiser]
	\label{prop:FSD_lambda}
	Fix a reference measure $P$ on $\R$ with cumulative distribution function~$F$
	and fix a tilting parameter $\lambda>0$.  Let $Q^*_{\lambda,a^*}$ be the maximizer of the problem \eqref{eq:robust-problem}. Then, $G^*_{\lambda,a^*}$ dominates $F$ in first order sense that is 
	\begin{equation}\label{eq:FSD_Qstar}
		F(x)\;\ge\;G_{\lambda,a^*}^*(x)\quad\text{for all }x\in\R .
	\end{equation}
\end{proposition}
\begin{proof}
Denote the normalizing constant
\[
  Z:=1-\bigl(1-e^{-1/\lambda}\bigr)F(a^*), 
  \qquad 0<e^{-1/\lambda}<1,\; 0\le F(a^*)\le1 .
\]
Then
\[
  e^{-1/\lambda}< Z\le 1
  \;\;\Longrightarrow\;\;
  0<\frac{e^{-1/\lambda}}{Z}\le 1,
  \qquad
  1\le\frac{1}{Z}<\infty,
\]
so the Radon–Nikodym derivative in~\eqref{eq:Qtilt} is \(<1\) on \((-\infty,a^*]\)
and \(>1\) on \((a^*,\infty)\). For any \(x\in\R\), 
\[
  G_{\lambda,a^*}^*(x)=
  \begin{cases}
    \displaystyle\frac{e^{-1/\lambda}}{Z}\,F(x), & x\le a^*,\\[8pt]
    \displaystyle\frac{e^{-1/\lambda}}{Z}\,F(a^*)
      +\frac{1}{Z}\bigl[F(x)-F(a^*)\bigr], & x>a^* ,
  \end{cases}
\]
because \( {\rm d}Q^*_{\lambda,a^*} = ({\rm d}Q^*_{\lambda,a^*}/{\rm d}P)\,{\rm d}P\).
\begin{enumerate}
\item If \(x\le a^*\), then \(G_{\lambda,a^*}^*(x)=\frac{e^{-1/\lambda}}{Z}F(x)\le F(x)\).
\item If \(x>a^*\), subtract the above expressions to obtain
      \[
        F(x)-G_{\lambda,a^*}^*(x)=\frac{1-e^{-1/\lambda}}{Z}\,F(a^*)\bigl[1-F(x)\bigr]\ge 0 .
      \]
\end{enumerate}

\smallskip
\noindent Hence \(F(x)\ge G_{\lambda,a^*}^*(x)\) for every \(x\in\R\); i.e.\ 
\(G^*_{\lambda,a^*}\) is first-order stochastically dominates \(F\),
establishing~\eqref{eq:FSD_Qstar}.
\end{proof}

\section{Potential applications}

\paragraph{Stress scenarios.}
Using the fixed–$\lambda$ tilt introduced above, we obtain a one-parameter family of distribution-wide stress measures; sampling or importance–weighting under the tilted law yields order-consistent scenarios whose severity is governed by $\lambda$ (equivalently, by the implied KL budget). One can use this as generator for adverse events across applications: (i) supervisory stress testing in banking/insurance \cite{BCBS2018,EIOPA2021}; (ii) robust portfolio selection and model-risk quantification \cite{GlassermanXu2014,BreuerCsiszar2016}; (iii) hedging and (re)insurance where Esscher-style exponential reweighting  \cite{Esscher1932, GerberShiu1994}; (iv) scenario mechanics grounded in importance sampling \cite{Glasserman2004}.

\paragraph{Reverse sensitivity.} Reverse sensitivity analysis  \cite{pesenti2019} stresses a \emph{scalar} risk functional (e.g., VaR or ES) and then selects, within a Kullback--Leibler (KL) neighbourhood of the baseline law. Instead, we impose a \emph{distribution-wide} stress by enforcing an order constraint on the \emph{entire} distribution via stochastic dominance type inequalities \eqref{eq:FSD_Qstar}. Given a target stress level \(p^\dagger\in(\tfrac12,1)\), we define the critical KL radius
\[
\varepsilon_{\mathrm{crit}}(p^\dagger)
:= \inf\Bigl\{\varepsilon>0:\ \max_{Q:\,D_{\mathrm{KL}}(Q\|P)\le \varepsilon}\ \min_{\gamma\in\Pi(P,Q)} \PP_{\gamma}(Y>X)\ \le\ p^\dagger\Bigr\},
\]
or, in the fixed–\(\lambda\) dual, the smallest \(\lambda\) such that \(\max_{a\in\R}\phi_\lambda(a)\le p^\dagger\). This formulation stresses families of tail events across all thresholds rather than a single tail-risk statistic, yielding distributional stress aligned with KL-ambiguity setup.

\printbibliography
\end{document}